\documentclass[10pt]{amsart}

\usepackage{fullpage} 
\usepackage{amssymb}
\usepackage{amsmath}
\usepackage{amsthm}

\newtheorem{theorem}{Theorem}[section]
\newtheorem{lemma}[theorem]{Lemma}
\newtheorem{proposition}[theorem]{Proposition}
\newtheorem{corollary}[theorem]{Corollary}
\newtheorem{conjecture}[theorem]{Conjecture}
\newtheorem{example}[theorem]{Example}
\theoremstyle{definition}

\newtheorem{definition}[theorem]{Definition}

\newtheorem{remark}[theorem]{Remark}

\newcommand{\N}{\mathbb{N}}

\newcommand{\U}{\mathcal{U}}
\newcommand{\W}{\mathcal{W}}
\newcommand{\Sub}[1]{\text{Sub($#1$)}}
\renewcommand{\subset}{\subseteq}
\newcommand{\la}{\langle}
\newcommand{\ra}{\rangle}
\renewcommand{\H}[1]{\text{H($#1$)}}
\newcommand{\minus}{\smallsetminus}
\renewcommand{\epsilon}{\varepsilon}
\renewcommand{\phi}{\varphi}
\newcommand{\trace}[2]{\text{trace}_{#2}\text{($#1$)}}
\newcommand{\acts}{\curvearrowright}

\author{Yair Glasner \and Daniel Kitroser}
\title{A note on LERF groups and generic group actions}
\date{}
\subjclass[2010]{Primary 20E26; Secondary 0B07}%
\keywords{LERF groups, amenable groups, amenable actions.}%

\begin{document}
\begin{abstract}
Let $G$ be a finitely generated group, $\Sub{G}$ the (compact, metric) space of all subgroups of $G$ with the Chaubuty topology and $X!$ the (Polish) group of all permutations of a countable set $X$. We show that the following  properties are equivalent: (i) Every finitely generated subgroup is closed in the profinite topology, (ii) the finite index subgroups are dense in $\Sub{G}$, (iii) A Baire generic homomorphism $\phi: G \rightarrow X!$ admits only finite orbits. Property (i) is known as the LERF property. We introduce a new family of groups which we call {\it{A-separable}} groups. These are defined by replacing, in (ii) above, the word ``finite index'' by the word ``co-amenalbe''. The class of A-separable groups contains all LERF groups, all amenable groups and more. We investigate some properties of these groups. 
\end{abstract}

\maketitle
\section{Introduction}
A group $G$ is said to be LERF (locally extended residually finite)  if every finitely generated subgroup can be separated from any element not contained in that subgroup by a homomorphism into a finite group. This is a strong form of residual finiteness which implies, for instance, that the word problem is solvable for $G$. Trivial examples for LERF groups include finite groups and finitely generated abelian groups. More interesting examples include finitely generated free groups (M. Hall, \cite{H}), surface groups (P. Scott, \cite{S}) and more generally, limit groups (H. Wilton, \cite{W}) as well as the Grigorchuk group (R. Grigorchuk and J. Willson \cite{GW}). More recently, the LERF property was proven for the fundamental group of a closed hyperbolic 3-manifold \cite{A}. In this paper we prove that the group-theoretic LERF property can be characterized via the properties of a Baire generic action of the group on an infinite, countable set. We also link the property of a group $G$ being LERF with the structure of $\Sub{G}$, the space of all subgroups of $G$, taken with the compact, metrizable Chaubuty topology. Explicitly, we show:
\begin{theorem}\label{LERF thm}
Let $G$ be a finitely generated group. Then the following conditions are equivalent:
\begin{enumerate}
\item $G$ is LERF.
\item The set of finite index subgroups of $G$ is dense in $\Sub{G}$.
\item  \label{itm:fin_orb_fg} For a Baire generic action of $G$ on an infinite, countable set all orbits are finite. 
\end{enumerate}
When $G$ is countable but not necessarily finitely generated, a similar equivalence holds, but condition (\ref{itm:fin_orb_fg}) above should be replaced by the somewhat weaker:
\begin{description}
\item[(3')] \label{itm:fin_orb} The set of actions all of whose orbits are finite is dense.
\end{description}
Even in this case, it is still true that generically every finitely generated subgroup $H < G$ admits only finite orbits. 
\end{theorem}
Precise definitions of a `generic action' and of the topological spaces involved in this paper are given in section 2. In section 3 we discuss the LERF property and prove Theorem \ref{LERF thm}.  

A subgroup $\Delta < \Gamma$ is called {\it{co-amenable}} if there is a $\Gamma$-invariant mean on $\Gamma/\Delta$. Co-amenable subgroups generalize finite index subgroup in much the same way that amenable groups generalize finite groups. In view of that and of Theorem \ref{LERF thm} we can generalize the notion of LERF groups as follows:
\begin{definition}
A group $G$ is \emph{amenably separable}, or \emph{A-separable} for short, if the set of co-amenable subgroups of $G$ is dense in $\Sub{G}$.
\end{definition} 
In view of the above thoerem every LERF group is $A$-separable. Another obvious example of $A$-separable groups are amenable groups since all of their subgroups are co-amenable. In this paper we initiate the study of A-separable groups. Our hope is that the notion of A-separability will prove to be a useful generalization of the, a-priory very different, properties of LERF and amenability. Similar perhaps to the way sofic groups generalize the the notions of residual finiteness and amenability. The analogue of Theorem \ref{LERF thm} is the following
\begin{theorem}\label{Asep thm}
A countable group $G$ is A-separable if and only if for a generic action of $G$ on , the action on every orbit is amenable.
\end{theorem}
That is, A-separability is an example of another property that can be characterized via the space of all actions of the group on some fixed countable set. Note that here, in contrast to the LERF property, we no longer see a difference between finitely generated and general countable groups. This is due to the fact that an action of a group is amenable if and only if it's restriction to any finitely generated subgroup is. Here are some properties of A-separable groups.
\begin{theorem} \label{basic_prop}
The following properties hold for the class of A-separable groups: 
\begin{itemize}
\item LERF groups and amenable groups are A-separable, 
\item The class of A-separable groups is closed under free products,
\item There exist A-separable groups which  are neither LERF nor amenable,
\item A group with property (T) is A-separable if and only if it is LERF.
\end{itemize}
\end{theorem}

\section{Basic definitions and facts}
Throughout this paper, $X$ is a countable set and $G$ is a countable group where ``countable" will always mean infinite and countable.

\begin{itemize}
\item $\Sub{G}$ will denote the space of all subgroups of $G$ equiped with the Chaubuty topology. Since $G$ is discrete and countable, a basis to the Chaubuty topology on $\Sub{G}$ can be given by the sets:
 \[ \W(H,\Omega) = \{ K\in\Sub{G}\mid K\cap \Omega = H\cap \Omega\}\ \big( H\in\Sub{G}, \Omega\subset G\text{ finite}. \big)\] $\Sub{G}$ is a compact metrizable space.

\item $X!$ denotes the symmetric group of $X$, endowed with the toplogy of pointwise convergence, or in other words, a basis for the topology is given by the sets 
\[ \U(\sigma,A) =\{ \tau\in X!\mid \tau a = \sigma a \text{ and } \tau^{-1} a = \sigma^{-1} a, \forall a\in A\}\ \big( \sigma\in X!, A\subset X\text{ is finite} \big).\] It is straightforward to show that $X!$ is a Polish group (see \cite{D} for more detail on this group). 

\item Let $\H{G} := \text{Hom}(G,X!)$ denote the set of all permutation representations (or \emph{actions}) of $G$ on $X$. $\H{G}$ is a closed subset of $X!^{|G|}$ (taken with the product topology) hence, it is a Polish space. Explicitly, a basis for the topology of $\H{G}$ is given by the sets:
\[  \mathcal{O}(\phi,S,A) = \{\sigma\in\H{G} \mid \forall s\in S\ \forall a\in A : \sigma(s)a = \phi(s)a \}\ \big(\phi\in\H{G}, S\subset G\text{ finite}, A\subset X\text{ finite}\big).\]

\item If $\sigma\in\H{G}$ and $x\in X$ then $G\overset{\sigma}{\acts} X$ denotes the $G$-action on $X$ defined by $\sigma$. $G_x(\sigma)$ will denote the stabilizer of $x$ in $G$ with respect to the action defined by $\sigma$, and $\sigma(G)x$ will denote the orbit. 
\end{itemize}
\begin{lemma}\label{lem1}
For every $x\in X$, the map:
\begin{align*}
&\H{G} \to \Sub{G}\\
&\sigma \mapsto G_x(\sigma)
\end{align*}
is continuous.
\end{lemma}
\begin{proof}
Let $\phi\in\H{G}$ and let $\Omega\subset G$ be a finite set. If $\sigma\in \mathcal{O}(\phi,\Omega,\{x\})$ then \[\forall g\in\Omega: g\in G_x(\sigma) \Leftrightarrow x = \sigma(g)x = \phi(g)x \Leftrightarrow g\in G_x(\phi)\] i.e $G_x(\sigma)\in\W(G_x(\phi),\Omega)$.
\end{proof}
\begin{definition}
Let $v$ be a word in $n$ variables, $x\in X, \varphi_1,\dots,\varphi_n \in X!$ and write $v(\varphi_1,\dots,\varphi_n) = v_k\cdots v_1$. The \emph{trace} of $x$ under $\bar{v} = v(\varphi_1,\dots,\varphi_n)$ is the ordered set of points:
\[ \text{trace}_{\bar{v}}(x) := \{ x, v_1 x,\dots, v_{k-1}v_{k-2}\cdots v_1 x, \bar{v}x\}. \]
\end{definition}
\begin{lemma}\label{lem2}
To a given $\sigma \in \H{G}$ there are arbitrarily close actions, with infinitely many fixed points. \end{lemma}
\begin{proof}
Given finite sets $S \subset G, Y \subset X$ we seek an action $\sigma' \in \mathcal{O}(\sigma,S,Y)$ with infinitely many fixed points. Consider the action $(\sigma, \operatorname{1}): G \acts X \sqcup \N$, obtained from $\sigma$ by augmenting countably many fixed points. The desired action $\sigma' = \phi^{-1} (\sigma,\operatorname{1}) \phi$ is obtained by intertwining this action via any bijection $\phi: X \to X \sqcup \N$ with the property that $\phi$ is the identity when restricted to $Y\cup \big(\bigcup_{s\in S} \sigma(s)Y\big)$.
\end{proof}
A subset of a Baire space is called \emph{generic} if it contains a countable intersection of dense, open sets. Given a Polish space $Y$ and a property P, we will say that ``a generic element of Y has property P" if P holds for all elements of some generic subset of Y.

\section{The LERF property}
\begin{definition}
A group $G$ is \emph{locally extended residually finite} or LERF, if for every finitely generated $K\leq G$ and $g\in G\minus K$ there exists a finite index subgroup of $G$ that contains $K$ but does not contain $g$.
LERF groups are also called \emph{subgroup separable} in the literature.
\end{definition}
\begin{proof}[proof of Theorem \ref{LERF thm}]
We will prove the the statement about countable groups, from which the statement about finitely generated groups directly follows. 

{\textit{(1) $\implies$ (2)}}. Let $L\in \Sub{G}$. We want to approximate $L$ by a finite index subgroup of $G$. We can assume that $L$ is finitely generated since the finitely generated subgroups are always dense in $\Sub{G}$. Indeed, given any finite subset $\Omega \subset G$, we have $\langle L \cap \Omega \rangle \in \W(L,\Omega)$. Now for every $g \in (G \minus L) \cap \Omega$, let $K_g\in \Sub{G}$ be the finite index subgroup, given by the LERF property, such that $L\subset K_g$ and $g \notin K_g$. Then $K = \bigcap_{g \in (G \setminus L) \cap\Omega} K_g$ is of finite index in $G$ and $K\in\W(L,\Omega)$.

{\textit{(2) $\implies$ (3)}}. Denote
\begin{align*}
& \Delta = \{\sigma \in \H{G} \ | \ \text{every f.g. subgroup has only finite orbits}\} = \bigcap_{x \in X, \ H \stackrel{\text{f.g.}}{<} G}\Delta_{H,x}  \\
& \Delta_{H,x} = \left \{ \left. \sigma\in\H{G} \ \right| \text{ for every f.g. } H<G, \left| \sigma(H)x \right| < \infty  \right \}\ (x\in X). 
\end{align*}
Since this is a countable intersection, by Baire's theorem, it is enough to fix a subgroup $H$ generated by a finite set $S$ and a point $x \in X$ and show that $\Delta_{H,x}$ is open and dense. Given $\phi \in \delta_{H,x}$ let us denote by $A = \phi(H)x$ the given finite orbit then $\phi \in \mathcal{O}(\phi,S,A) \subset \Delta_{H,x}$ is an open neighborhood, proving that $\Delta_{H,x}$ is open. For the density statement we need to show that $\tilde{\Delta} := \left\{ \sigma \in \H{G} \ | \ \sigma(G) {\text{ has only finite orbits}} \right\} \subset \Delta$ is dense in $\H{G}$. 
Let $\tau\in\H{G}$ be such that the $\tau$-orbit of $x$ is infinite and $L=G_x(\tau)$ be the corresponding stabilizer; we seek $\sigma \in \mathcal{O}(\tau,S,A) \cap \tilde{\Delta}$ where $S \subset G$ and $A \subset X$ are given finite subsets. After possibly enlarging $A$ we may assume without loss of generality that $x \in A$. We also assume without loss of generality that $\tau$ is a transitive action, since we can approximate $\tau$ on the orbit $\tau(G)x$, leaving all other orbits untouched. 

Let $\Omega \subset G$ be a finite set such that $A \subset \tau(\Omega)x$ and let $\widetilde{\Omega}\subset G$ be a finite symmetric subset containing the identity element such that $\Omega \cup \big(\bigcup_{s\in S} s\Omega\big)\subset \widetilde{\Omega}$. By hypothesis, there exists $K\in\Sub{G}$ of finite index in $G$ such that $K\cap (\widetilde{\Omega}\cdot\widetilde{\Omega}) = L\cap (\widetilde{\Omega}\cdot\widetilde{\Omega})$. This implies that the function given by $gK\overset{f}{\mapsto} \tau(g)x$ is well defined and one-to-one on the set $\widetilde{\Omega}K = \{gK\ |\ g\in \widetilde{\Omega}\} \subset G/K$. Now, extend $f$ to an injective map $\tilde{f}: G/K \rightarrow X$, and choose an action $\sigma\in\H{G}$ such that $\tilde{f}$ becomes a $G$-map with respect to the (left) regular action of $G$ on $G/K$. This can be done for example by setting
\[ \sigma(g)y =
\begin{cases}
\tilde{f}g\tilde{f}^{-1}y, & y\in\tilde{f}(G/K).\\
y, &\text{otherwise}.
\end{cases}\]
It is easy to verify that $K = G_x(\sigma)$ thus the orbit of $x$ under $\sigma(G)$ is finite. All other orbits are just singletons. Now for every $s\in S$ and $y \in A$, by our choice of $\Omega$ we have $y = \tau{\omega} x $ for some $\omega \in \Omega$. Thus
$$\sigma(s) y = \tilde{f}s \tilde{f}^{-1} y = f s f^{-1} y = f s f^{-1} \tau(\omega) x = f s \omega K = \tau(s \omega) x = \tau(s) y.$$ Note that the exact same proof shows that the collection $\{\phi \in \H{G} \ | \ {\text{ all orbits of $G$ are finite}}\}$ is dense in $\H{G}$, thus concluding the proof. 

{\textit{(3) $\implies$ (1)}}. Let $S\subset G$ be finite, $L = \la S\ra$ and $g\in G\minus L$. We wish to find a finite index subgroup $K\in\Sub{G}$ such that $L\leq K$ and $g\notin K$. If $L$ is of finite index in $G$ we are finished so assume $[G:L] = \infty$. Let $\phi: G\to (G/L)!$ be the regular representation of $G$ on $G/L$. By hypothesis, there exists a representation $\psi: G\to (G/L)!$ with only finite orbits such that $\forall s\in S: \psi(s)L = \phi(s)L = sL = L$ and $\psi(g)L = \phi(g)L = gL \neq L$. The stabilizer $G_L(\psi)$ is of finite index in $G$, contains $S$ (and thus contains $L$) and does not contain $g$.
\end{proof}
\begin{remark}\label{remark1}
In the proof of the last theorem, the proof that (2) implies (3) shows in fact that if $\mathcal{D}\subset\Sub{G}$ is dense in $\Sub{G}$ then for any $x\in X$, the set $\{\sigma\in\H{G}\mid G_x(\sigma)\in\mathcal{D}\}$ is dense in $\H{G}$.
\end{remark}
In order to demonstrate the use of the Theorem \ref{LERF thm} we give a short proof that finitely generated free groups are LERF using the language of generic permutation representations:

\begin{proposition} \label{prop:FnLERF}
$F_n$, the free group on $n$ generators is LERF.
\end{proposition}
\begin{proof}
First notice that $\H{F_n} = X!^n$ and so it is enough to show that 
\[\Sigma := \{(\sigma_1,\dots,\sigma_n)\in X!^n\mid \langle \sigma_1,\dots,\sigma_n\rangle\text{ has only finite orbits}\}\]
is generic in $X!^n$. We first show that $\Sigma$ is dense in $X!$. Let $(\varphi_1,\dots,\varphi_n)\in X!^n$ and $A\subset X$ be finite. For $\sigma\in X!$ we define the \emph{support} of $\sigma$ as: $\text{supp}(\sigma) = \{ x\in X \mid \sigma x\neq x\}$. For every $i=1,\dots,n$ there exists $\sigma_i\in X!$ such that $\sigma_i\big|_{A} = \varphi_i\big|_{A}$ and $\sigma_i$ has finite support. To see this notice that since $|A|= |\varphi_i(A)|$ we have
that $|\varphi_i(A)\minus A| = |A\minus\varphi_i(A)|$ so we can define $\sigma_i$ to coinside with $\varphi_i$ on $A$, take $\varphi_i(A)\minus A$ bijectively onto $A\minus\varphi_i(A)$ and fix every other point. Finally we notice that every orbit of $\langle \sigma_1,\dots,\sigma_n\rangle$ is finite: if $a\in \bigcup_i \text{supp}(\sigma_i)$ then the orbit of $a$ is contained in the finite set $\bigcup_i \text{supp}(\sigma_i)$ and otherwise, $a$ is fixed by $\langle \sigma_1,\dots,\sigma_n\rangle$.

We now prove that $\Sigma$ is $G_{\delta}$. Consider the sets 
\[\Sigma_x :=  \{(\sigma_1,\dots,\sigma_n)\in X!^n\mid \text{the orbit of $x$ under }\langle \sigma_1,\dots,\sigma_n\rangle\text{ is finite}\},\ (x\in X).\]
Obviously, $\Sigma = \bigcap_{x\in X}\Sigma_x$. Also, every $\Sigma_x$ is open: let $(\sigma_1,\dots,\sigma_n)\in\Sigma_x$ and let $A$ be the finite orbit of $x$ under $\langle \sigma_1,\dots,\sigma_n\rangle$. Then 
\[ \{ (\psi_1,\dots,\psi_n)\in X!^n\ |\ \forall i: \psi_i\big|_A =  \sigma_i\big|_A,\ \psi_i^{-1}\big|_A =  \sigma_i^{-1}\big|_A\} \]
is an open neighborhood of $(\sigma_1,\dots,\sigma_n)$ that is contained in $\Sigma_x$.
\end{proof}
\begin{example}
The free group $F_{\infty}$ on a countable number of generators is LERF. But a generic action of $F_{\infty}$ on a countable set $X$ is transitive. 
\end{example}
\begin{proof}
The proof that $F_{\infty}$ is LERF - and hence satisfies all the conditions (1),(2),(3') of Theorem \ref{basic_prop} is similar to that given for finitely generated free groups in Proposition \ref{prop:FnLERF} - we leave the details to the reader. 

Now by Baire's theorem, in order to prove transitivity it is enough to show that the set $\Theta(x,y) = \{ \phi \in \H{F_{\infty}} \ | \ y \in \phi(F_{\infty})x \}$ is open and dense. Indeed for $\phi \in \Theta(x,y)$ there is a specific element $\omega \in F_{\infty}$ such that $\phi(\omega)x = y$ and the inclusion $\phi \in \mathcal{O}(\phi,\{\omega\},\{x\}) \subset \Theta(x,y)$ proves that this set is open. 

Let $F_{\infty} = \langle x_1,x_2,\ldots \rangle$ be a free generating set for $F_{\infty}$ and let $F_r = \langle x_1,x_2,\ldots,x_r \rangle$ be the free group generated by the first $r$-generators.  For the density note that if $S \subset F_{\infty}$ and $A \subset X$ are finite then in fact $S \subset F_r$ for any large enough value of $r$. Now if $\mathcal{O}(\phi, S,A)$ is the basic open set defined by these choices then we can define $\sigma \in \mathcal{O}(\phi,S,A)$ by setting $\sigma(x_i) = \phi(x_i)$ for every $1 \le i \le r$ and then defining $\sigma(x_{r+1})$ in such a way that $\sigma(x_{r+1})x = y$. 
\end{proof}

\section{A-separability}\label{Asep}
\begin{definition}
An action $G \acts Y$ of a discrete, countable group $G$ is called \emph{amenable} if it satisfies any one of the following equivalent conditions:
\begin{itemize}
\item For every $\epsilon > 0$ and $\Omega\subset G$ finite, $Y$ admits an \emph{$(\epsilon,\Omega)$-F\o lner} subset, that is, a finite set $F\subset Y$ such that $\dfrac{|gF\Delta F|}{|F|} < \epsilon$ for all $g\in \Omega$.
\item There exists a finitely additive $G$-invariant probability measure on $Y$. 
\end{itemize}
When the action is transitive, of the form $G \acts G/K$ these conditions are further equivalent to the following
\begin{itemize}
\item If $G$ acts continuously on a compact space and $K$ admits an invariant Borel measure, then so does $G$. 
\end{itemize}
\end{definition}
\noindent In the transitive case it is sometimes convenient to adopt group theoretic terminology as follows:\begin{definition}
A subgroup $K$ of a group $G$ is called \emph{co-amenable} if the quasiregular action $G \acts G/K$ is amenable. \end{definition}
\noindent The equivalence of these three conditions is classical. We recall the following 

\begin{remark}
A F\o lner-sequence can be chosen to be increasing (with respect to inclusion).
\end{remark}

As mentioned in the introduction, LERF groups and amenable groups are A-separable, i.e. the co-amenable groups are dense in the space of all subgroups. We now prove Theorem \ref{Asep thm}, giving a characterization of A-separability in the language of generic actions:
\begin{proof}[proof of Theorem \ref{Asep thm}]
For $x\in X$ denote $\Sigma(x) =\{\sigma\in\H{G}\mid G\overset{\sigma}{\acts} \sigma(G)x\text{ is amenable}\}$ and $\Sigma = \cap_{x  \in X} \Sigma(x)$. If $\Sigma$ is generic then $\Sigma(x)$ dense in $\H{G}$ and by Lemma \ref{lem1} the image of this set $\{ G_x(\sigma)\mid \sigma\in\Sigma(x)\}$ is a dense subset of $\Sub{G}$ consisting of co-amenable subgroups.  

Conversely, assume that the set of co-amenable subgroups is dense in $\Sub{G}$ and we wish to prove that the set $\Sigma$ is generic in $\H{G}$. It is enough to show that $\Sigma(x)$ is generic in $\H{G}$ for every $x\in X$. The density of
$\Sigma(x)$ is assured by the hypothesis and Remark \ref{remark1}. To show that $\Sigma(x)$ is $G_{\delta}$, it is enough to show that the condition that a specific finite set $F \subset \sigma(G)x$ is $(\epsilon,\Omega)$-F\o lner is open, where $\epsilon > 0$ and $\Omega \subset G$ is finite. Assume this is the case for some $\sigma \in \H{G}$ we seek an open neighborhood $\sigma \in \mathcal{O} \subset \H{G}$ such that $F$ is still contained in the orbit, and is still F\o lner for every $\phi \in \mathcal{O}$. For every $f \in F$ pick a group element $g_f\in G$ such that $f = \sigma(g_f)x$. Let $F' = F\cup \{x\}$ and $\Omega' = \Omega \cup \{g_f \ | \ f \in F\}$ - the desired neighborhood is given by $\mathcal{O} = \mathcal{O}(\sigma,F',\Omega')$. 
\end{proof}

LERF and amenable groups are not the only examples of A-separable groups. In order to give an example of an A-separable group which is neither LERF nor amenable, we first prove that A-separability is closed to taking free products:

\begin{proposition}\label{Asep product}
Let $G$ and $K$ be countable groups. If $G$ and $K$ are A-separable then so is $G * K$.
\end{proposition}
\begin{proof}
Every element of $\H{G * K}$ is of the form $\phi * \psi$ for $\phi\in\H{G}, \psi\in\H{K}$; where $\phi * \psi$ is defined by setting $\big(\phi * \psi\big)(g) = \phi(g)$ and $\big(\phi * \psi\big)(k) = \psi(k), \forall g\in G, k \in K$ and expending the definition to the free product.

For every $x\in X,\ \epsilon > 0$ and finite subsets $S\subset G,\ T\subset K$ let 
\[\Sigma(x,\epsilon,S,T) = \{ \sigma*\tau\in\H{G * K} \mid \text{the $(\sigma *\tau)$-orbit of $x$ contains an $(\epsilon,S\cup T)$-F\o lner set}\}.\] We want to prove that $\Sigma = \bigcap \Sigma(x,\frac{1}{n},S,T)\ \big( x\in X,n\in\N,S\subset G, T\subset K \text{ finite}\big)$ is generic in $\H{G * K}$ and since $X, G$ and $K$ are countable, it is enough to show that the sets $\Sigma(x,\epsilon,S,T)$ are open and dense for every $x\in X,\ \epsilon > 0$ and finite subsets $S\subset G,\ T\subset K$. The argument that shows that $\Sigma(x,\epsilon,S,T)$ is open was given in the proof of Theorem \ref{Asep thm}.

Fix $x, \epsilon, S$ and $T$ as above. We prove that $\Sigma(x,\epsilon,S,T)$ is dense in $\H{G * K}$. Let $\phi *\psi\in\H{G*K}$ and let $A\subset X$ be finite. We will find $\phi' \in\H{G}$ and $\psi' \in\H{K}$ such that $\phi'(s)a = \phi(s)a,\ \psi'(t)a =\psi(t)a$ for all $s\in S, t\in T, a\in A$ and such that $\phi' *\psi' \in \Sigma(x,\epsilon,S,T)$. We can assume that $x\in A$. By A-separability, there exist $\sigma \in\H{G}$ and $\tau \in \H{K}$ such that $\sigma(s)a = \phi(s)a,\ \tau(t)a =\psi(t)a$ for all $s\in S, t\in T, a\in A$ and such that the actions $G \overset{\sigma}{\acts} X$ and $K\overset{\tau}{\acts}X$ are amenable on every orbit. Let $L := \sigma *\tau(G*K) = \la \sigma(G),\tau(K)\ra$.

\emph{Case 1: all the $\sigma$ and $\tau$ orbits which are contained in $Lx$ are finite}. Let $B\subset Lx$ be a finite, $\sigma$-invariant set containing $A\cap Lx$ and let $C = \bigcup_{b\in B} \tau(K)b$. We define a representation $\phi' \in \H{G}$ by declaring every $c\in C \minus B$ and every element in the $\sigma$-orbit of $c$ to be a fixed point for $\phi'$ and on every other element of $X$, $\phi'(g)$ identifies with $\sigma(g)$ for all $g\in G$. Notice that since $B$ is $\sigma$-invariant, $\phi'(g)$ is well defined and acts that same as $\sigma(g)$ on $B$ for all $g\in G$. In particular, $\phi'(g)$ agrees with $\phi(g)$ on $A$. We have that $C$ is finite, invariant  under both $\phi'$ and $\tau$ and contains $x$. Setting $\psi' = \tau$, the $(\phi'*\psi')$-orbit of $x$ is finite so the orbit itself is an $(\epsilon,S\cup T)$-F\o lner set for $\phi'*\psi'$.

\emph{Case 2: $Lx$ contains an infinite $\sigma$-orbit or an infinite $\tau$-orbit}. Assume, without loss of generality, that $Lx$ contains an infinite $\tau$-orbit $Y$, denote $B = A\cup\big( \bigcup_{s\in S} \sigma(s)A\big)$ and let $F_n$ be an increasing F\o lner-sequence in $Y$ for the $\tau$-action. Since the sets $F_n$ are finite, non of them is $\tau$-invariant and so the F\o lner-sequence does not stabilize. This implies that $|F_n| \to \infty$ and in particular, $Y$ contains an $(\epsilon,T)$-F\o lner set $F$ such that $|F| > \dfrac{2(|B|+1)}{\epsilon}$. Now, let $z\in G*K$ be such that $(\sigma *\tau)(z)x\in F$ and such that $z$ is of minimal length with respect to the canonical presentation: $z = g_n k_n g_{n-1} k_{n-1} \cdots g_1 k_1\ (g_i\in G, k_j\in K, g_1,\dots,g_{n-1},k_2,\dots,k_n \neq 1)$. Denote $y = (\sigma *\tau)(z)x$. By lemma \ref{lem2}, we can assume that $\sigma$ has infinitely many fixed points. In particular, there exits a set $C\subset X$ on which $\sigma(G)$ acts trivially such that $|C| = |F\minus(B\cup\{y\})|$ and such that $C$ does not intersect the finite set $B\cup F\cup \trace{x}{(\sigma *\tau)(z)}$ where we think of $(\sigma *\tau)(z)$ as the word over $X!$ corresponding to the given presentation of $z$. Denote $D = F\minus(B\cup\{y\})$ and let $\xi \in X!$ be a permutation of order $2$ that takes $C$ bijectively onto $D$ and acts trivially on $X\minus (C\cup D)$. We define an action $\phi' \in\H{G}$ by $\phi'(g) = \xi^{-1}\sigma(g)\xi$ for all $g\in G$. Since $\xi$ acts trivially on $B$ we have that $\forall s\in S,\ \forall a\in A: \phi'(s)a = \sigma(s)a = \phi(s)a$ and that every element of $D$ is fixed under $\phi'(s)$ for all $s\in S$, hence:
\[ \forall s\in S: \dfrac{|\phi'(s)F\Delta F|}{|F|} \leq \dfrac{2|F\minus D|}{|F|} \leq \dfrac{2(|B|+1)}{|F|} < \epsilon.\]
Thus $F$ is $(\epsilon,S)$-F\o lner for $\phi'$ and $(\epsilon,T)$-F\o lner for $\psi':=\tau$ and thus $F$ is $(\epsilon,S\cup T)$-F\o lner for $\phi' * \psi'$. Notice that by the minimality of the length of $z$ we have that $\trace{x}{(\sigma *\tau)(z)}\cap F = \{y\}$ and so $\xi$ acts trivially on $\trace{x}{(\sigma *\tau)(z)}$. This means that $(\phi' *\psi')(z)x =y\in F$ and since $F$ is contained in a $\tau$-orbit this implies that $F$ is contained in the $(\phi' *\psi')$-orbit of $x$, as required.

\end{proof}
Recall that the $(m,n)$ Baumslag-Solitar group is the group $BS(m,n) = \la s,t\mid t^{-1}s^m t = s^n\ra$. It is well known that $BS(m,n)$ is solvable (hence amenable) if and only if $m=1$.

\begin{proposition}
For every $n$, the group $BS(1,n)$ is not LERF.
\end{proposition}
\begin{proof}
Write: $BS(1,n) = \la s,t\mid t^{-1}s t = s^n\ra$ and notice that $t^{-1}\la s\ra t = \la s^n\ra \subsetneqq \la s\ra$. Thus, an element of $\la s\ra \minus t^{-1}\la s\ra t$ cannot be separated from $t^{-1}\la s\ra t$ by a homomorphism into a finite group.
\end{proof}

\begin{corollary}
There exist non-LERF, non-amenable A-separable groups.
\end{corollary}
\begin{proof}
Let $G = BS(1,n)$ for some $n$. $G$ is amenable hence A-separable and so, by Proposition \ref{Asep product}, $G * G$ is A-separable. On the other hand, $G * G$ is not LERF since $G$ is not LERF and LERF passes to subgroups. $G * G$ is also not amenable since it contains a free subgroup on two generators.
\end{proof}

To conclude all the statements promised in the introduction we prove the following:
\begin{proposition}
A group $G$ with Kazhdan property (T) is A-separable if and only if it is LERF. 
\end{proposition}
\begin{proof}
This follows directly form the fact that a transitive action $G \acts G/H$ is amenable if and only if $G/H$ is finite. The argument for that follows directly from property (T). If this action is amenable and $F \subset G/H$ is an $(K,\epsilon)$ F\o lner set then $1_{F} \in \ell^2(G/H)$ is a $(K,\epsilon)$-almost invariant vector. Taking $(K,\epsilon)$ to be Kazhdan constants for $G$ we can deduce the existence of a non-zero invariant vector $f \in \ell^2(G/H)$. Since the action of $G$ on $G/H$ is transitive $f$ must be a constant function. But a non-zero constant function is in $\ell^2$ if and only if $G/H$ is finite. 
\end{proof}
We conjecture that even this cannot happen in a non-trivial way namely
\begin{conjecture}
A countable group $G$ with Kazhdan property (T) is LERF if and only if it is finite. 
\end{conjecture}

This work was written while Y.G. was on sabbatical at the University of Utah. I am  very grateful to hospitality of the math department there and to the NSF grants that enabled this visit. Y.G. acknowledges support from U.S. National Science Foundation grants DMS 1107452, 1107263, 1107367 ``RNMS: Geometric structures And Representation varieties" (the GEAR Network). Both authors were partially supported by the Israel Science Foundation grant ISF 441/11.

\end{document}